\documentclass[12pt]{article}
\usepackage[utf8]{inputenc}
\usepackage[font=small,labelfont=bf]{caption}
\usepackage{amssymb, amsmath, amsthm, mathrsfs, calligra}
\usepackage{tikz}
\usepackage{color}
\usepackage{graphicx}
\usepackage[a4paper, width=150mm, top=25mm, bottom=25mm]{geometry}
\graphicspath{ {images/} }
\usetikzlibrary{positioning, calc, arrows, decorations.pathmorphing, backgrounds, fit, petri}

\newtheorem{thm}{Theorem}[section]
\newtheorem{dfn}[thm]{Definition}

\newtheorem{prop}[thm]{Proposition}
\newtheorem{lemma}[thm]{Lemma}

\begin{document}

\title{\textbf{Computable Isomorphisms for Certain Classes of Infinite Graphs}}
\author{Hakim J. Walker\thanks{Partially supported by NSF Grant DMS-1202328.}}
\date{\small Department of Mathematics\\
\small The George Washington University\\
\small Washington, DC 20052\\
\small hjwalker@gwu.edu\\
October 17, 2016}

\maketitle

\begin{abstract}
We investigate (2,1):1 structures, which consist of a countable set $A$ together with a function $f: A \to A$ such that for every element $x$ in $A$, $f$ maps either exactly one element or exactly two elements of $A$ to $x$. These structures extend the notions of injection structures, 2:1 structures, and (2,0):1 structures studied by Cenzer, Harizanov, and Remmel, all of which can be thought of as infinite directed graphs. We look at various computability-theoretic properties of (2,1):1 structures, most notably that of computable categoricity. We say that a structure $\mathcal{A}$ is computably categorical if there exists a computable isomorphism between any two computable copies of $\mathcal{A}$. We give a sufficient condition under which a (2,1):1 structure is computably categorical, and present some examples of (2,1):1 structures with different computability-theoretic properties. 
\end{abstract}

\section{Introduction}
In computable model theory, we study the properties of classical mathematical structures from the perspective of computability theory.  A set $X$ is \emph{computable} if there is a Turing program, or more generally, an algorithm, that can decide the membership of $X$. Furthermore, a set $X$ is \emph{computably enumerable} if there is an algorithm to enumerate the elements of $X$. A countable structure $\mathcal{A}$ over a finite language is \emph{computable} if its domain is computable and all of its functions and relations are computable. Unless otherwise specified, all of our structures are computable, and we assume that their domain is $\omega$, the set of natural numbers.

One of the key concepts in computable model theory is that of computable isomorphisms between structures. We say that two computable structures $\mathcal{A}$ and $\mathcal{B}$ that are isomorphic to each other are \emph{computably isomorphic} if there exists a computable function $h: \omega \to \omega$, where $h$ is an isomorphism from $\mathcal{A}$ to $\mathcal{B}$. Computable isomorphisms preserve not only the functions and relations of a structure, but also the algorithmic properties of the structure. It is very possible for two isomorphic computable structures to not be \emph{computably} isomorphic. Thus, we have the following definition.\\

\begin{dfn}
A computable structure $\mathcal{A}$ is \emph{computably categorical} if every two computable isomorphic copies of $\mathcal{A}$ are computably isomorphic.
\end{dfn}

We seek to classify computable structures up to computable isomorphism. That is, within a class of structures, we wish to provide characterizations of those computable structures that are computably categorical. This has been done for various classes of mathematical structures. For example, Goncharov and Dzgoev \cite{Autostability}, and independently Remmel \cite{linearOrders}, proved that a computable linear order is computably categorical if and only if it has only finitely many successor pairs. Additionally, Goncharov and Dzgoev \cite{Autostability}, LaRoche \cite{booleanAlgebras1}, and Remmel \cite{booleanAlgebras2} independently proved that a computable Boolean algebra is computably categorical if and only if it has only finitely many atoms. Goncharov, Lempp, and Solomon \cite{abelianGroups} characterized computably categorical ordered abelian groups as those with finite rank. In \cite{equivalenceStructures}, Calvert, Cenzer, Harizanov, and Morozov gave a characterization of computably categorical equivalence structures (structures consisting of a countable set and an equivalence relation).

Computable categoricity has also been extensively studied for certain types of graphs. Lempp, McCoy, Miller, and Solomon \cite{treesFiniteHeight} characterized computable trees of finite height that are computably categorical, and Miller \cite{treesInfiniteHeight} previously showed that no computable tree of infinite height is computably categorical. In \cite{SLFG}, Csima, Khoussainov, and Liu investigated computable categoricity of \emph{strongly locally finite graphs}, those which have countably many finite connected components, by looking at proper embeddability of the components. As an example, it can be shown that the first graph in Figure 1 is computably categorical, whereas the second graph is not.

\begin{center}
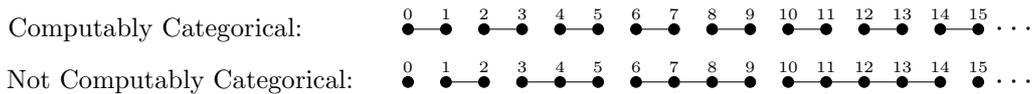

\captionof{figure}{Examples of strongly locally finite graphs.}
\medskip

\begin{tikzpicture}[inner sep=0.5mm, vertex/.style={circle,draw,fill=black}, every label/.style={black, font=\tiny}, highlight/.style={circle,draw,fill=black}]

\node[vertex](A'') at (0,0.7) [label=above: 0]{};
\node[vertex](B'') at (0.5,0.7) [label=above: 1]{};
\node[vertex](C'') at (1,0.7) [label=above: 2]{};
\node[vertex](D'') at (1.5,0.7) [label=above: 3]{};
\node[vertex](E'') at (2,0.7) [label=above: 4]{};
\node[vertex](F'') at (2.5,0.7) [label=above: 5]{};
\node[vertex](G'') at (3,0.7) [label=above: 6]{};
\node[vertex](H'') at (3.5,0.7) [label=above: 7]{};
\node[vertex](I'') at (4,0.7) [label=above: 8]{};
\node[vertex](J'') at (4.5,0.7) [label=above: 9]{};
\node[vertex](K'') at (5,0.7) [label=above: 10]{};
\node[vertex](L'') at (5.5,0.7) [label=above: 11]{};
\node[vertex](M'') at (6,0.7) [label=above: 12]{};
\node[vertex](N'') at (6.5,0.7) [label=above: 13]{};
\node[vertex](O'') at (7,0.7) [label=above: 14]{};
\node[vertex](P'') at (7.5,0.7) [label=above: 15]{};
\node at (8,0.7) {$\cdots$};
\node at (-3.33,0.7) {\footnotesize{Computably Categorical:}};
 
\draw[-] (A'') to (B'');
\draw[-] (C'') to (D'');
\draw[-] (E'') to (F'');
\draw[-] (G'') to (H'');
\draw[-] (I'') to (J'');
\draw[-] (K'') to (L'');
\draw[-] (M'') to (N'');
\draw[-] (O'') to (P'');

\node[highlight](A') at (0,0) [label=above: 0]{};
\node[highlight](B') at (0.5,0) [label=above: 1]{};
\node[highlight](C') at (1,0) [label=above: 2]{};
\node[highlight](D') at (1.5,0) [label=above: 3]{};
\node[highlight](E') at (2,0) [label=above: 4]{};
\node[highlight](F') at (2.5,0) [label=above: 5]{};
\node[highlight](G') at (3,0) [label=above: 6]{};
\node[highlight](H') at (3.5,0) [label=above: 7]{};
\node[highlight](I') at (4,0) [label=above: 8]{};
\node[highlight](J') at (4.5,0) [label=above: 9]{};
\node[highlight](K') at (5,0) [label=above: 10]{};
\node[highlight](L') at (5.5,0) [label=above: 11]{};
\node[highlight](M') at (6,0) [label=above: 12]{};
\node[highlight](N') at (6.5,0) [label=above: 13]{};
\node[highlight](O') at (7,0) [label=above: 14]{};
\node[highlight](P') at (7.5,0) [label=above: 15]{};
\node at (8,0) {$\cdots$};
\node at (-3,0) {\footnotesize{Not Computably Categorical:}};
 
\draw[-] (B') to (C');
\draw[-] (D') to (E');
\draw[-] (E') to (F');
\draw[-] (G') to (H');
\draw[-] (H') to (I');
\draw[-] (I') to (J');
\draw[-] (K') to (L');
\draw[-] (L') to (M');
\draw[-] (M') to (N');
\draw[-] (N') to (O');

\end{tikzpicture}
\end{center}
\medskip

We are interested in classes of infinite directed graphs that are derived from computable functions. Cenzer, Harizanov, and Remmel first studied directed graphs of this type in \cite{Injections}, where they defined injection structures. An \emph{injection structure} $\mathcal{A}=(A,f)$ is a countable set $A$ together with an injective function $f: A \to A$. An injection structure can be completely classified up to isomorphism by the number, type, and size of its \emph{orbits}, which are the types of connected components the structure may have. 

\begin{center}
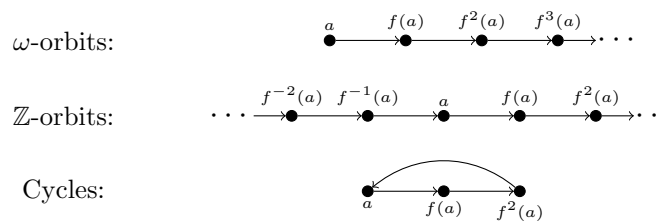

\captionof{figure}{Types of orbits in an injection structure.}
\begin{tikzpicture}[inner sep=0.5mm, vertex/.style={circle,draw,fill=black}, every label/.style={black, font=\tiny}, highlight/.style={circle,draw,fill=black}]

\node[vertex](A) at (-1.5,1) [label=above: $a$]{};
\node[vertex](B) at (-0.5,1) [label=above: $f(a)$]{};
\node[vertex](C) at (0.5,1) [label=above: $f^2(a)$]{};
\node[vertex](D) at (1.5,1) [label=above: $f^3(a)$]{};
\node at (2.3,1) {$\cdots$};
\node at (-5,1) {\footnotesize{$\omega$-orbits:}};

\draw[->] (A) to (B);
\draw[->] (B) to (C);
\draw[->] (C) to (D);
\draw[->] (D) to (2,1);

\node[vertex](A') at (-2,0) [label=above: $f^{-2}(a)$]{};
\node[vertex](B') at (-1,0) [label=above: $f^{-1}(a)$]{};
\node[vertex](C') at (0,0) [label=above: $a$]{};
\node[vertex](D') at (1,0) [label=above: $f(a)$]{};
\node[vertex](E') at (2,0) [label=above: $f^2(a)$]{};
\node at (2.8,0) {$\cdots$};
\node at (-2.8,0) {$\cdots$};
\node at (-5,0) {\footnotesize{$\mathbb{Z}$-orbits:}};

\draw[->] (A') to (B');
\draw[->] (B') to (C');
\draw[->] (C') to (D');
\draw[->] (D') to (E');
\draw[->] (E') to (2.5,0);
\draw[->] (-2.5,0) to (A');

\node[vertex](A'') at (-1,-1) [label=below: $a$]{};
\node[vertex](B'') at (0,-1) [label=below: $f(a)$]{};
\node[vertex](C'') at (1,-1) [label=below: $f^2(a)$]{};
\node at (-5,-1) {\footnotesize{Cycles:}};

\draw[->] (A'') to (B'');
\draw[->] (B'') to (C'');
\draw[->] (C'') to [bend right=40] (A'');

\end{tikzpicture}
\end{center}

Cenzer, Harizanov, and Remmel obtained the following characterization theorem for injection structures.

\begin{thm}
A computable injection structure $\mathcal{A}$ is computably categorical if and only if $\mathcal{A}$ has only finitely many infinite orbits, that is, only finitely many $\omega$-orbits and only finitely many $\mathbb{Z}$-orbits.
\end{thm}

Next, Cenzer, Harizanov, and Remmel \cite{Two-to-One} looked at \emph{two-to-one} (2:1) \emph{structures} $\mathcal{A}=(A,f)$, where $|f^{-1}(a)|=2$ for all $a \in A$, as well as (2,0):1 \emph{structures} $\mathcal{A}=(A,f)$, where $|f^{-1}(a)|\in \{0,2\}$ for all $a \in A$. Here, $|X|$ denotes the cardinality of the set $X$. Thus, in a 2:1 structure, every element has exactly two pre-images under $f$, and in a (2,0):1 structure, every element has either exactly two pre-images or no pre-image under $f$. The types of orbits for these structures are shown below.

\begin{center}
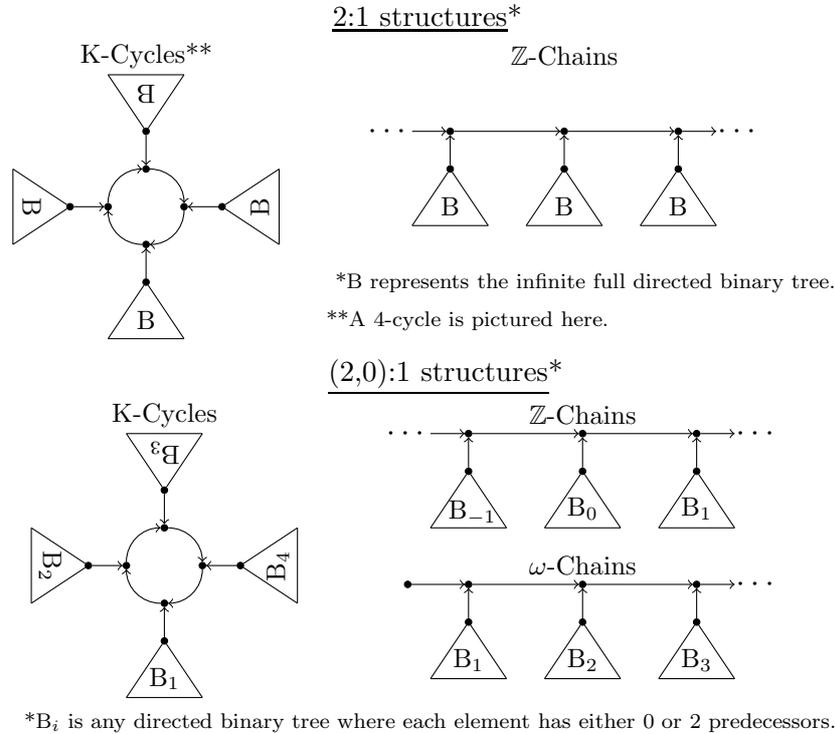

\captionof{figure}{Types of orbits in 2:1 structures and (2,0):1 structures.}
\begin{tikzpicture}[inner sep=0.3mm, vertex/.style={circle,draw,fill=black}, every label/.style={black, font=\scriptsize}, highlight/.style={circle,draw,fill=black}]

\node at (3.7,2.5) {\small \underline{2:1 structures}*};
\node[vertex](A) at (0,0.5) {};
\node[vertex](B) at (0.5,0){};
\node[vertex](C) at (0,-0.5) {};
\node[vertex](D) at (-0.5,0) {};
\node [vertex](A1) at (0,1){};
\node[vertex](B1) at (1,0){};
\node[vertex](C1) at (0, -1){};
\node[vertex](D1) at (-1,0){};
\node at (0, 1.5){\footnotesize{\rotatebox[origin=c]{180}{B}}};
\node at (1.5, 0){\footnotesize{\rotatebox[origin=c]{90}{B}}};
\node at (0, -1.5){\footnotesize{B}};
\node at (-1.5, 0){\footnotesize{\rotatebox[origin=c]{-90}{B}}};
\node at (0,2) {\footnotesize{K-Cycles**}};

\draw [->] (A) to [bend left=45] (B);
\draw [->] (B) to [bend left=45] (C);
\draw [->] (C) to [bend left=45] (D);
\draw [->] (D) to [bend left=45] (A);

\draw [->] (A1) to (A);
\draw [-] (-0.5, 1.75) to (A1);
\draw [-] (0.5,1.75) to (A1);
\draw [-] (-0.5,1.75) to (0.5,1.75);

\draw[->] (B1) to (B);
\draw[-] (1.75, -0.5) to (B1);
\draw[-] (1.75, 0.5) to (B1);
\draw[-] (1.75, 0.5) to (1.75, -0.5);

\draw[->] (C1) to (C);
\draw[-] (-0.5,-1.75) to (C1);
\draw[-] (0.5,-1.75) to (C1);
\draw[-] (-0.5,-1.75) to (0.5,-1.75);

\draw[->] (D1) to (D);
\draw[-] (-1.75, -0.5) to (D1);
\draw[-] (-1.75,0.5) to (D1);
\draw[-] (-1.75,-0.5) to (-1.75,0.5);

\node[vertex](E) at (4,1){};
\node[vertex](F) at (5.5,1){};
\node[vertex](G) at (7,1){};
\node at (7.8,1){$\cdots$};
\node at (3.2,1) {$\cdots$};
\node[vertex](E1) at (4,0.5){};
\node[vertex](F1) at (5.5,0.5){};
\node[vertex](G1) at (7,0.5){};
\node at (4, 0){\footnotesize{B}};
\node at (5.5, 0){\footnotesize{B}};
\node at (7, 0){\footnotesize{B}};
\node at (5.5,2) {\footnotesize{$\mathbb{Z}$-Chains}};

\draw[->] (E) to (F);
\draw[->] (F) to (G);
\draw[->] (E1) to (E);
\draw[->] (F1) to (F);
\draw[->] (G1) to (G);
\draw[->] (G) to (7.5,1);
\draw[->] (3.5,1) to (E);
\draw[-] (3.5,-.25) to (E1);
\draw[-] (4.5,-.25) to (E1);
\draw[-] (3.5,-.25) to (4.5,-.25);

\draw[-] (5,-.25) to (F1);
\draw[-] (6,-.25) to (F1);
\draw[-] (5,-.25) to (6,-.25);

\draw[-] (6.5,-.25) to (G1);
\draw[-] (7.5,-.25) to (G1);
\draw[-] (6.5,-.25) to (7.5,-.25);

\node at (5.78,-1) {\scriptsize{*B represents the infinite full directed binary tree.}};
\node at (4.23, -1.5) {\scriptsize{**A 4-cycle is pictured here.}};

\end{tikzpicture}
\medskip

\begin{tikzpicture}[inner sep=0.3mm, vertex/.style={circle,draw,fill=black}, every label/.style={black, font=\scriptsize}, highlight/.style={circle,draw,fill=black}]

\node at (3.7,2.5) {\small \underline{(2,0):1 structures}*};
\node[vertex](A) at (0,0.5) {};
\node[vertex](B) at (0.5,0){};
\node[vertex](C) at (0,-0.5) {};
\node[vertex](D) at (-0.5,0) {};
\node [vertex](A1) at (0,1){};
\node[vertex](B1) at (1,0){};
\node[vertex](C1) at (0, -1){};
\node[vertex](D1) at (-1,0){};
\node at (0, 1.5){\footnotesize{\rotatebox[origin=c]{180}{B$_3$}}};
\node at (1.5, 0){\footnotesize{\rotatebox[origin=c]{90}{B$_4$}}};
\node at (0, -1.5){\footnotesize{B$_1$}};
\node at (-1.5, 0){\footnotesize{\rotatebox[origin=c]{-90}{B$_2$}}};
\node at (0,2) {\footnotesize{K-Cycles}};

\draw [->] (A) to [bend left=45] (B);
\draw [->] (B) to [bend left=45] (C);
\draw [->] (C) to [bend left=45] (D);
\draw [->] (D) to [bend left=45] (A);

\draw [->] (A1) to (A);
\draw [-] (-0.5, 1.75) to (A1);
\draw [-] (0.5,1.75) to (A1);
\draw [-] (-0.5,1.75) to (0.5,1.75);

\draw[->] (B1) to (B);
\draw[-] (1.75, -0.5) to (B1);
\draw[-] (1.75, 0.5) to (B1);
\draw[-] (1.75, 0.5) to (1.75, -0.5);

\draw[->] (C1) to (C);
\draw[-] (-0.5,-1.75) to (C1);
\draw[-] (0.5,-1.75) to (C1);
\draw[-] (-0.5,-1.75) to (0.5,-1.75);

\draw[->] (D1) to (D);
\draw[-] (-1.75, -0.5) to (D1);
\draw[-] (-1.75,0.5) to (D1);
\draw[-] (-1.75,-0.5) to (-1.75,0.5);

\node[vertex](E) at (4,1.75){};
\node[vertex](F) at (5.5,1.75){};
\node[vertex](G) at (7,1.75){};
\node at (7.8,1.75){$\cdots$};
\node at (3.2,1.75) {$\cdots$};
\node[vertex](E1) at (4,1.25){};
\node[vertex](F1) at (5.5,1.25){};
\node[vertex](G1) at (7,1.25){};
\node at (4.05, 0.70){\footnotesize{B$_{-1}$}};
\node at (5.5, 0.75){\footnotesize{B$_0$}};
\node at (7, 0.75){\footnotesize{B$_1$}};
\node at (5.5,2) {\footnotesize{$\mathbb{Z}$-Chains}};

\draw[->] (E) to (F);
\draw[->] (F) to (G);
\draw[->] (E1) to (E);
\draw[->] (F1) to (F);
\draw[->] (G1) to (G);
\draw[->] (G) to (7.5,1.75);
\draw[->] (3.5,1.75) to (E);
\draw[-] (3.5,0.5) to (E1);
\draw[-] (4.5,0.5) to (E1);
\draw[-] (3.5,0.5) to (4.5,0.5);

\draw[-] (5,0.5) to (F1);
\draw[-] (6,0.5) to (F1);
\draw[-] (5,0.5) to (6,0.5);

\draw[-] (6.5,0.5) to (G1);
\draw[-] (7.5,0.5) to (G1);
\draw[-] (6.5,0.5) to (7.5,0.5);

\node[vertex](H) at (4,-0.25){};
\node[vertex](I) at (5.5,-0.25){};
\node[vertex](J) at (7,-0.25){};
\node at (7.8,-0.25){$\cdots$};
\node[vertex] at (3.2,-0.25) {};
\node[vertex](H1) at (4,-0.75){};
\node[vertex](I1) at (5.5,-0.75){};
\node[vertex](J1) at (7,-0.75){};
\node at (4, -1.25){\footnotesize{B$_1$}};
\node at (5.5, -1.25){\footnotesize{B$_2$}};
\node at (7, -1.25){\footnotesize{B$_3$}};
\node at (5.5,0){\footnotesize{$\omega$-Chains}};

\draw[->] (H) to (I);
\draw[->] (I) to (J);
\draw[->] (H1) to (H);
\draw[->] (I1) to (I);
\draw[->] (J1) to (J);
\draw[->] (J) to (7.5,-0.25);
\draw[->] (3.2,-0.25) to (H);

\draw[-] (3.5,-1.5) to (H1);
\draw[-] (4.5,-1.5) to (H1);
\draw[-] (3.5,-1.5) to (4.5,-1.5);

\draw[-] (5,-1.5) to (I1);
\draw[-] (6,-1.5) to (I1);
\draw[-] (5,-1.5) to (6,-1.5);

\draw[-] (6.5,-1.5) to (J1);
\draw[-] (7.5,-1.5) to (J1);
\draw[-] (6.5,-1.5) to (7.5,-1.5);

\node at (3.5, -2.08){\scriptsize{*B$_i$ is any directed binary tree where each element has either 0 or 2 predecessors.}};

\end{tikzpicture}
\end{center}

Cenzer, Harizanov, and Remmel investigated computably categorical (2,0):1 structures by considering additional structural and algorithmic properties. They also characterized the computably categorical 2:1 structures by proving the following theorem.\\

\begin{thm}
A computable 2:1 structure $\mathcal{A}$ is computably categorical if and only if $\mathcal{A}$ has only finitely many $\mathbb{Z}$-chains.
\end{thm}
\medskip

In this paper, we define a (2,1):1 structure, which is a natural extension of the structures introduced by Cenzer, Harizanov, and Remmel. Our ultimate goal is to provide a characterization of computable categoricity for these directed graphs, as has been done for the graphs discussed above. In Section 2, we establish fundamental structural and computability-theoretic properties of (2,1):1 structures, and use these to investigate computable categoricity for such structures. In section 3, we present some examples of (2,1):1 structures with certain desired computability-theoretic properties. 

\section{Computable Categoricity of (2,1):1 Structures}

We begin this section by defining a (2,1):1 structure.

\begin{dfn}
A \emph{(2,1):1 structure} $\mathcal{A}=(A,f)$ is a set $A$ together with a function $f:A \to A$ such that $|f^{-1}(a)|\in \{1,2\}$ for all $a \in A$. That is, every element in $A$ has either exactly two pre-images or exactly one pre-image under $f$.
\end{dfn}

\noindent Naturally, we say that a (2,1):1 structure $\mathcal{A} = (A, f)$ is \emph{computable} if $A$ is a computable set and $f$ is a computable function. From now on, we will assume that all of our (2,1):1 structures are computable, with $A=\omega$, unless otherwise stated.

Although we have mentioned the concept in the introduction, we must now formally define the \emph{orbit of an element} in a (2,1):1 structure.

\begin{dfn}
Let $\mathcal{A} = (A, f)$ be a (2,1):1 structure, and let $x \in A$.The \emph{orbit of $x$ in $\mathcal{A}$}, denoted by $\mathcal{O}_{\mathcal{A}} (x)$, is defined as follows:\\

\centering{$\mathcal{O}_{\mathcal{A}} (x) = \{y \in A | (\exists m,n)(f^m(x) = f^n(y))\}$}
\end{dfn}

\noindent Here, $f^m(x)$ denotes the result of iterating the function $f(x)$ \emph{m} times on $x$. If we think of (2,1):1 structures as directed graphs, we can think of orbits as the connected components of the graph.

It is not hard to see that a (2,1):1 structure can only consist of two general types of orbits. We refer to them as \textbf{K-cycles} and \textbf{$\mathbb{Z}$-chains}, following the naming conventions for the orbits of 2:1 structures used by Cenzer, Harizanov, and Remmel. We describe these orbits below.\\

\large
\noindent{\textbf{K-Cycles}}

\normalsize
A \textbf{K-cycle} is a directed cycle with $k$ elements, where every element in the cycle has a directed binary tree attached, each of which is either infinite or empty.

\begin{center}
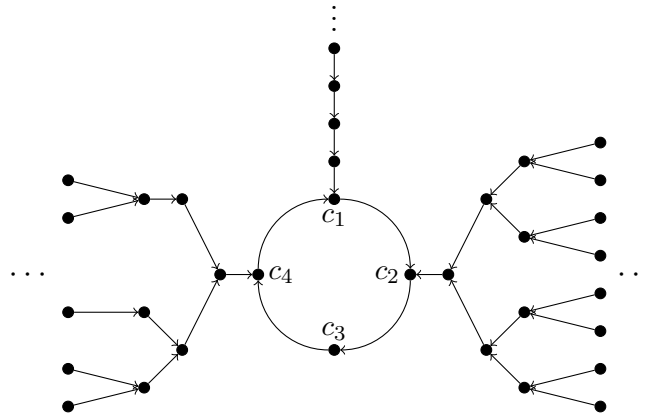

\captionof{figure}{An example of a 4-cycle.}
\begin{tikzpicture}[inner sep=0.5mm, vertex/.style={circle,draw,fill=black}, every label/.style={black, font=\small}, highlight/.style={circle,draw,fill=black}]

\node[vertex](A) at (0,1) [label=below:$c_1$]{};
\node[vertex](B)  at (1,0)  [label = left:$c_2$] {};
\node[vertex](C)   at (0,-1)  [label=above:$c_3$]{};
\node[vertex](D)   at (-1,0)  [label = right:$c_4$]{};

\node[vertex](A1) at (0,1.5){};
\node[vertex](A2) at (0,2){};
\node[vertex](A3) at (0,2.5){};
\node[vertex](A4) at (0,3){};
\node(A5) at (0,3.5){$\vdots$};

\node[vertex](B1) at (1.5,0){};
\node[vertex](B21) at (2,-1){};
\node[vertex](B22) at (2, 1){};
\node[vertex](B31) at (2.5,-1.5){};
\node[vertex](B32) at (2.5, -0.5){};
\node[vertex](B33) at (2.5, 0.5){};
\node[vertex](B34) at (2.5, 1.5){};
\node[vertex](B41) at (3.5, -1.75){};
\node[vertex](B42) at (3.5, -1.25){};
\node[vertex](B43) at (3.5, -0.75){};
\node[vertex](B44) at (3.5, -0.25){};
\node[vertex](B45) at (3.5, 0.25){};
\node[vertex](B46) at (3.5, 0.75){};
\node[vertex](B47) at (3.5, 1.25){};
\node[vertex](B48) at (3.5, 1.75){};
\node (B5) at (4,0){$\cdots$};

\node[vertex](D1) at (-1.5,0){};
\node[vertex](D21) at (-2, 1){};
\node[vertex](D22) at (-2, -1){};
\node[vertex](D31) at (-2.5, 1){};
\node[vertex](D32) at (-2.5, -0.5){};
\node[vertex](D33) at (-2.5, -1.5){};
\node[vertex](D41) at (-3.5, 1.25){};
\node[vertex](D42) at (-3.5, 0.75){};
\node[vertex](D43) at (-3.5, -0.5){};
\node[vertex](D44) at (-3.5, -1.25){};
\node[vertex](D45) at (-3.5, -1.75){};
\node(D5) at (-4,0){$\cdots$};

\draw [->] (A) to [bend left=45] (B);
\draw [->] (B) to [bend left=45] (C);
\draw [->] (C) to [bend left=45] (D);
\draw [->] (D) to [bend left=45] (A);

\draw [->] (A1) to (A);
\draw [->] (A2) to (A1);
\draw [->] (A3) to (A2);
\draw [->] (A4) to (A3);

\draw[->] (B1) to (B);
\draw[->] (B21) to (B1);
\draw[->] (B22) to (B1);
\draw[->] (B31) to (B21);
\draw[->] (B32) to (B21);
\draw[->] (B33) to (B22);
\draw[->] (B34) to (B22);
\draw[->] (B41) to (B31);
\draw[->] (B42) to (B31);
\draw[->] (B43) to (B32);
\draw[->] (B44) to (B32);
\draw[->] (B45) to (B33);
\draw[->] (B46) to (B33);
\draw[->] (B47) to (B34);
\draw[->] (B48) to (B34);

\draw[->] (D1) to (D);
\draw[->] (D21) to (D1);
\draw[->] (D22) to (D1);
\draw[->] (D31) to (D21);
\draw[->] (D32) to (D22);
\draw[->] (D33) to (D22);
\draw[->] (D41) to (D31);
\draw[->] (D42) to (D31);
\draw[->] (D43) to (D32);
\draw[->] (D44) to (D33);
\draw[->] (D45) to (D33);

\end{tikzpicture}
\end{center}

An element $x$ of a K-cycle is called a \textbf{cyclic element} if there exists an $n>0$ such that $f^n(x)=x$. We denote the cyclic elements of a K-cycle by $c_1$, $c_2$,...,$c_K$, where $c_i \not= c_j$ for $1 \leq i < j \leq K$, and $f(c_r)=c_{r+1}$ for $1 \leq r < K$, and $f(c_K)=c_1$. Since each K-cycle consists of only one directed cycle, we can uniquely specify a particular K-cycle within a (2,1):1 structure by listing its K cyclic elements.

In Figure 4, each of the cyclic elements $c_1$, $c_2$, $c_3$, $c_4$ has a different type of binary tree attached. The tree attached to $c_1$ is often referred to as a \textbf{degenerate tree}, where every element in the tree has exactly one pre-image. The tree attached to $c_2$ is a \textbf{full binary tree}, as every element has exactly two pre-images. The tree attached to $c_3$ is the \textbf{empty tree}, and the tree attached to $c_4$ is an arbitrary infinite binary tree that is neither empty, degenerate, nor full.\\

\large
\noindent{\textbf{$\mathbb{Z}$-Chains}}

\normalsize
A \textbf{$\mathbb{Z}$-chain} consists of a $\mathbb{Z}$-orbit of elements, where every element in the orbit has a directed binary tree attached, each of which is either infinite or empty. 

\begin{center}
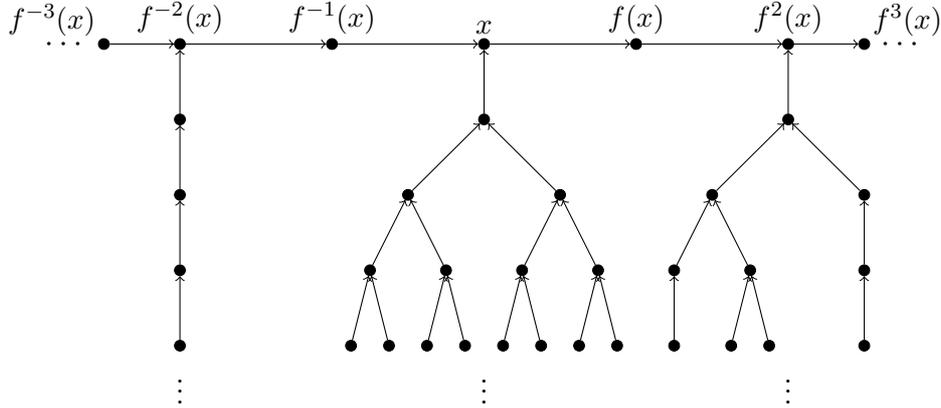

\captionof{figure}{An example of a $\mathbb{Z}$-chain.}
\begin{tikzpicture}[inner sep=0.5mm, vertex/.style={circle,draw,fill=black}, every label/.style={black, font=\small}]

\node[vertex](A) at (1,0)[label=above left:$f^{-3}(x)$]{};
\node[vertex](B) at (2,0)[label=above:$f^{-2}(x)$]{};
\node[vertex](C) at (4,0)[label=above:$f^{-1}(x)$]{};
\node[vertex](D) at (6,0)[label=above:$x$]{};
\node[vertex](E) at (8,0)[label=above:$f(x)$]{};
\node[vertex](F) at (10,0)[label=above:$f^2(x)$]{};
\node[vertex](G) at (11,0)[label=above right:$f^3(x)$]{};
\node(X) at (0.5,0){$\cdots$};
\node(Y) at (11.5,0){$\cdots$};

\node[vertex](B1) at (2,-1){};
\node[vertex](B2) at (2,-2){};
\node[vertex](B3) at (2,-3){};
\node[vertex](B4) at (2,-4){};
\node(B5) at (2,-4.5){$\vdots$};

\node[vertex](D1) at (6,-1){};
\node[vertex](D21) at (5, -2){};
\node[vertex](D22) at (7, -2){};
\node[vertex](D31) at (4.5, -3){};
\node[vertex](D32) at (5.5, -3){};
\node[vertex](D33) at (6.5, -3){};
\node[vertex](D34) at (7.5, -3){};
\node[vertex](D41) at (4.25,-4){};
\node[vertex](D42) at (4.75,-4){};
\node[vertex](D43) at (5.25,-4){};
\node[vertex](D44) at (5.75,-4){};
\node[vertex](D45) at (6.25,-4){};
\node[vertex](D46) at (6.75,-4){};
\node[vertex](D47) at (7.25,-4){};
\node[vertex](D48) at (7.75,-4){};
\node(D5) at (6, -4.5){$\vdots$};

\node[vertex](F1) at (10,-1){};
\node[vertex](F21) at (9,-2){};
\node[vertex](F22) at (11,-2){};
\node[vertex](F31) at (8.5, -3){};
\node[vertex](F32) at (9.5,-3){};
\node[vertex](F33) at (11,-3){};
\node[vertex](F41) at (8.5,-4){};
\node[vertex](F42) at (9.25,-4){};
\node[vertex](F43) at (9.75,-4){};
\node[vertex](F44) at (11,-4){};
\node(F5) at (10,-4.5){$\vdots$};

\draw[->] (A) to (B);
\draw[->] (B) to (C);
\draw[->] (C) to (D);
\draw[->] (D) to (E);
\draw[->] (E) to (F);
\draw[->] (F) to (G);

\draw[->] (B1) to (B);
\draw[->] (B2) to (B1);
\draw[->] (B3) to (B2);
\draw[->] (B4) to (B3);

\draw[->] (D1) to (D);
\draw[->] (D21) to (D1);
\draw[->] (D22) to (D1);
\draw[->] (D31) to (D21);
\draw[->] (D32) to (D21);
\draw[->] (D33) to (D22);
\draw[->] (D34) to (D22);
\draw[->] (D41) to (D31);
\draw[->] (D42) to (D31);
\draw[->] (D43) to (D32);
\draw[->] (D44) to (D32);
\draw[->] (D45) to (D33);
\draw[->] (D46) to (D33);
\draw[->] (D47) to (D34);
\draw[->] (D48) to (D34);

\draw[->] (F1) to (F);
\draw[->] (F21) to (F1);
\draw[->] (F22) to (F1);
\draw[->] (F31) to (F21);
\draw[->] (F32) to (F21);
\draw[->] (F33) to (F22);
\draw[->] (F41) to (F31);
\draw[->] (F42) to (F32);
\draw[->] (F43) to (F32);
\draw[->] (F44) to (F33);

\end{tikzpicture}
\end{center}

Here, a $\mathbb{Z}$-orbit refers to an infinite set of elements $\{..., f^{-2}(x), f^{-1}(x), x, f(x), f^2(x),...\}$ such that for all $m,n \in \mathbb{Z}$ with $m \not= n$, $f^m(x) \not= f^n(x)$. Unlike with cyclic elements in a K-cycle, a $\mathbb{Z}$-orbit within a $\mathbb{Z}$-chain does not necessarily \emph{uniquely} determine the $\mathbb{Z}$-chain, since a $\mathbb{Z}$-chain may contain more than one different $\mathbb{Z}$-orbit. Indeed, if a $\mathbb{Z}$-chain contains any element with two pre-images, then that $\mathbb{Z}$-chain will contain more than one distinct $\mathbb{Z}$-orbit. However, given a $\mathbb{Z}$-chain, we can establish a \textbf{canonical $\mathbb{Z}$-orbit} $\{..., f^{-2}(x), f^{-1}(x), x, f(x), f^2(x),...\}$, where $x$ is the least element in the $\mathbb{Z}$-chain (under the usual ordering on $\mathbb{N}$), and $f^{-(n+1)}(x)$ is the least pre-image of $f^{-n}(x)$ for all $n \geq 0$. Thus, in Figure 5, if we take the labeled elements to be the canonical $\mathbb{Z}$-orbit of the $\mathbb{Z}$-chain, then $f^{-1}(x)$ is the least pre-image of $x$, $f^{-3}(x)$ is the least pre-image of $f^{-2}(x)$, and so on.

As we can see from Figures 4 and 5, the orbits of a (2,1):1 structure are essentially directed graphs. However, this is not quite correct, as the orbit of an element $x$ is only defined to be the \emph{set} of elements in the same connected component as $x$, and does not include any additional structure specifying an edge relation. It is advantageous to be able to refer to a connected component in a (2,1):1 structure as a directed graph instead of just a set of vertices. So we formalize this notion in the following definition.

\begin{dfn}
Let $\mathcal{A} = (A,f)$ be a (2,1):1 structure, and let $x \in A$. The \emph{connected component of $x$ in $\mathcal{A}$}, denoted by $C_{\mathcal{A}}(x)$, is the directed graph associated with $\mathcal{O}_{\mathcal{A}} (x)$. That is:

\begin{center}
$C_{\mathcal{A}}(x) = (V,E)$
\end{center}

where $V = \mathcal{O}_{\mathcal{A}} (x)$ and $E = \{(x, f(x)): x \in V\}$.
\end{dfn}

To further analyze our structures as graphs, we explore another fundamental property: the \emph{tree of an element}. 

\begin{dfn}
Let $\mathcal{A} = (A,f)$ be a (2,1):1 structure, and let $x \in A$. The \emph{tree of $x$ in $\mathcal{A}$}, denoted by $tree_{\mathcal{A}}(x)$, is defined as:

\begin{center}
$tree_{\mathcal{A}}(x) = \{a \in A: (\exists n)(f^n(a)=x)\}$
\end{center}

Furthermore, the \emph{Tree of $x$ in $\mathcal{A}$}, denoted by $Tree_{\mathcal{A}}(x)$, is the directed graph associated with $tree_{\mathcal{A}}(x)$. That is:

\begin{center}
$Tree_{\mathcal{A}}(x) = (V,E)$
\end{center}

where $V=tree_{\mathcal{A}}(x)$ and $E=\{(x,f(x)): x \in V \wedge f(x) \in V\}$.
\end{dfn}

Intuitively, we can think of $tree_{\mathcal{A}}(x)$ as the set of all \emph{predecessors} of $x$ (or the set of all elements that will eventually \emph{lead} to $x$), and we can think of $Tree_{\mathcal{A}}(x)$ as a rooted binary tree with $x$ as its root. It is apparent that if $c_i$ is a cyclic element, then $tree_{\mathcal{A}}(c_i) = \mathcal{O}_{\mathcal{A}} (c_i)$, which is the entire K-cycle containing $c_i$. However, we often wish to refer to those elements in a K-cycle that are connected to a cyclic element via a directed path that does not contain other cyclic elements. So we introduce the notion of an \emph{exclusive tree}.

\begin{dfn}
Let $\mathcal{A} = (A,f)$ be a (2,1):1 structure, and let $c_i$ be a cyclic element on a K-cycle in $A$. The \emph{exclusive tree of $c_i$ in $\mathcal{A}$}, denoted by $extree_{\mathcal{A}}(c_i)$, is the following set:
\begin{center}
$extree_{\mathcal{A}}(c_i) = \{a \in A: (\exists n)[f^n(a)=c_i \wedge (\forall m<n)(f^{m+K}(a) \not= f^m(a))]\}$
\end{center}
The \emph{exclusive Tree of $c_i$ in $\mathcal{A}$}, denoted by $exTree_{\mathcal{A}}(c_i)$ is the directed graph associated with $extree_{\mathcal{A}}(c_i)$.
\end{dfn}
\bigskip

The following properties of (2,1):1 structures will also be useful later.\\

\begin{dfn}
Let $\mathcal{A} = (A,f)$ be a (2,1):1 structure, $x \in A$, and $n \in \omega$.
\begin{itemize}
\item[(a)] The \emph{nth level of the tree of $x$}, denoted by $tree_{\mathcal{A}}(x|n)$, is defined as:
\begin{center}
$tree_{\mathcal{A}}(x|n) = \{a \in A: f^n(a)=x\}$
\end{center}
Similarly, the \emph{nth level of the exclusive tree of $c_i$} is defined as:
\begin{center}
$extree_{\mathcal{A}}(c_i|n) = \{a \in A: a \in extree_{\mathcal{A}}(c_i) \wedge f^n(a)=c_i\}$
\end{center}
\item[(b)] The \emph{tree of $x$ truncated at level n}, denoted by $tree_{\mathcal{A}}(x,n)$, is defined as:
\begin{center}
$tree_{\mathcal{A}}(x,n) = \{a \in A: (\exists m \leq n)(f^m(a)=x)\}$
\end{center}
Similarly, the \emph{exclusive tree of $c_i$ truncated at level n} is defined as:
\begin{center}
$extree_{\mathcal{A}}(c_i,n) = \{a \in A: a \in extree_{\mathcal{A}}(c_i) \wedge (\exists m \leq n)(f^m(a)=x)\}$
\end{center}
\end{itemize}
\end{dfn}

\noindent Naturally, $Tree_{\mathcal{A}}(x,n)$ and $exTree_{\mathcal{A}}(c_i,n)$ are the associated directed graphs for the sets described in Definition 2.6(b). 

Finally, we introduce two special functions for (2,1):1 structures, the \emph{branching function} and the \emph{branch isomorphism function}, which will allow us to further study the computable categoricity of our graphs.

\begin{dfn}
Let $\mathcal{A} = (A,f)$ be a \emph{(2,1):1} structure. The \emph{branching function of $\mathcal{A}$}, denoted by $\beta_{\mathcal{A}}: \mathbb{N} \to \{1,2\}$, is defined as:

$$\beta_{\mathcal{A}}(x)=
		\begin{cases}
			1 &\text{if } (\forall x_1, x_2)(f(x_1)=f(x_2)=x \implies x_1=x_2),\\
			2 &\text{if } (\exists x_1, x_2)(f(x_1)=f(x_2)=x \wedge x_1 \not= x_2).
		\end{cases}$$
The \emph{hair set of $\mathcal{A}$}, denoted by $I_\mathcal{A}$, is defined as:

\begin{center}
{$I_\mathcal{A} = \{x \in A: \beta_{\mathcal{A}}(x) = 1\}$}
\end{center}
The \emph{split hair set of $\mathcal{A}$}, denoted by $\Lambda_\mathcal{A}$, is defined as:

\begin{center}
$\Lambda_\mathcal{A} = \{x \in A: \beta_{\mathcal{A}}(x) = 2\}$
\end{center}
\end{dfn}

\begin{dfn}
Let $\mathcal{A} = (A,f)$ be a (2,1):1 structure, and let $x \in \Lambda_{\mathcal{A}}$ have distinct pre-images $x_1$ and $x_2$. The \emph{branch isomorphism function of $\mathcal{A}$}, denoted by $iso_{\mathcal{A}}: \Lambda_{\mathcal{A}} \to \{0,1\}$, is defined as:

$$iso_{\mathcal{A}}(x)=
		\begin{cases}
			0 &\text{if } Tree_\mathcal{A} (x_1) \not\cong Tree_\mathcal{A} (x_2),\\
			1 &\text{if } Tree_\mathcal{A} (x_1) \cong Tree_\mathcal{A} (x_2).
		\end{cases}$$
\end{dfn}

Essentially, the branching function takes an element $x \in A$ as an input, and outputs the number of immediate predecessors of $x$. The branch isomorphism function takes an element $x$ with two distinct pre-images as an input, and tells us if the \emph{Trees} of those pre-images are isomorphic to each other. Note that if $c_1 \in \Lambda_{\mathcal{A}}$ is a cyclic element, then $iso{\mathcal{A}} (c_1) = 0$. This is because one predecessor of $c_1$ will be another cyclic element $c_K$ while the other predecessor will be a non-cyclic element $a$, so $Tree_\mathcal{A} (c_K)$ will be the entire K-cycle containing $c_1$, while $Tree_\mathcal{A} (a)$ will be an infinite binary tree with no cycles. Hence, $Tree_\mathcal{A} (c_K)$ is clearly not isomorphic to $Tree_\mathcal{A} (a)$.

It is also important to note that neither $\beta_{\mathcal{A}}$ nor $iso_{\mathcal{A}}$ is necessarily a computable function, even if the underlying (2,1):1 structure is computable. (In Section 3, we will construct an example of such a structure.) Also, while the domain of $\beta_{\mathcal{A}}$ is always computable, the domain of $iso_{\mathcal{A}}$ may not be computable. In fact, if $\beta_{\mathcal{A}}$ is not computable, then neither is the domain of $iso_{\mathcal{A}}$. We generally avoid this issue by assuming that the branching function is computable. However, computability of the branching function does not guarantee computability of the branch isomorphism function, as we will see in Section 3.

We will now establish our first result regarding computable isomorphisms between (2,1):1 structures.

\begin{lemma}
Let $\mathcal{A} = (A,f)$ and $\mathcal{B} = (B,g)$ be two computable isomorphic (2,1):1 structures, both with a computable branching function and a computable branch isomorphism function. If $a_0 \in A$ and $b_0 \in B$ are non-cyclic elements such that $Tree_\mathcal{A} (a_0) \cong Tree_\mathcal{B} (b_0)$, then the two Trees are computably isomorphic. Likewise, if $c_1 \in A$ and $d_1 \in B$ are cyclic elements such that $exTree_\mathcal{A} (c_1) \cong exTree_\mathcal{B} (d_1)$, then the two exclusive Trees are computably isomorphic.
\end{lemma}

\begin{proof}
We construct a computable isomorphism $h$ from $Tree_\mathcal{A} (a_0)$ to $Tree_\mathcal{B} (b_0)$ in stages as follows.\\

\emph{Stage 0}: Define $h_0(a_0) = b_0$.

\emph{Stage s+1}: Suppose that from stage $s$ we have $h_s$, an isomorphism from $Tree_\mathcal{A} (a_0, s)$ to $Tree_\mathcal{B} (b_0, s)$. For all elements $w \in tree_\mathcal{A}(a_0,s)$, define $h_{s+1}(w)=h_s(w)$. Let $x \in tree_\mathcal{A} (a_0|s)$ and let $y = h_s(x)$. If $\beta_{\mathcal{A}}(x)=1$, find the unique pre-image of $x$ under $f$, call it $x_1$, and find the pre-image of $y$ under $g$, call it $y_1$. Then define $h_{s+1}(x_1)=y_1$. 

If $\beta_{\mathcal{A}}(x)=2$, find both pre-images of $x$ under $f$, call them $x_1$ and $x_2$, and then find both pre-images of $y$ under $g$, call them $y_1$ and $y_2$. If $iso_\mathcal{A}(x) = 1$, then define $h_{s+1}(min\{x_1,x_2\})=min\{y_1,y_2\}$ and $h_{s+1}(max\{x_1,x_2\}) = max\{y_1,y_2\}$, where $min$ and $max$ are defined under the usual ordering on $\mathbb{N}$. If $iso_\mathcal{A}(x) = 0$, then there exists a level $n$ such that $Tree_\mathcal{A}(x_1,n) \not\cong Tree_\mathcal{A}(x_2,n)$. In that case, use $\beta_{\mathcal{A}}$ to reveal the vertices and edges of $Tree_\mathcal{A}(x_1)$ and $Tree_\mathcal{A}(x_2)$ one level at a time until we find such a level $n$. Then, use $\beta_{\mathcal{B}}$ to reveal the vertices and edges of $Tree_\mathcal{B}(y_1,n)$ and $Tree_\mathcal{B}(y_2,n)$. If $Tree_\mathcal{A}(x_1,n) \cong Tree_\mathcal{B}(y_1,n)$, then define $h_{s+1}(x_1)=y_1$ and $h_{s+1}(x_2)=y_2$. Otherwise, define $h_{s+1}(x_1)=y_2$ and $h_{s+1}(x_2)=y_1$. 

Repeat the procedure above for all $x \in tree_\mathcal{A} (a_0|s)$, so $h_{s+1}$ is defined on all elements in $tree_\mathcal{A} (a_0|s+1)$. This completes the construction. Let $h = lim_s h_s$.

We must now verify that $h$ is a computable isomorphism from $Tree_\mathcal{A} (a_0)$ to $Tree_\mathcal{B} (b_0)$.\\

\textbf{Claim 1}. The function $h$ is an isomorphism from $Tree_\mathcal{A} (a_0)$ to $Tree_\mathcal{B} (b_0)$.\\

\emph{Proof of Claim 1}. Suppose $h_s$ is an isomorphism from $Tree_\mathcal{A} (a_0,s)$ to $Tree_\mathcal{B} (b_0,s)$ such that there exists an isomorphism $H$ from $Tree_\mathcal{A} (a_0)$ to $Tree_\mathcal{B} (b_0)$ with $H\restriction_{tree_\mathcal{A}(a_0,s)}=h_s$. Let $x \in tree_\mathcal{A} (a_0|s)$ and let $y = h_s(x)$. 

If $\beta_{\mathcal{A}}(x)=1$, then it must be the case that $\beta_{\mathcal{B}}(y)=1$ as well. This means that $x$ and $y$ each have a unique pre-image, $x_1$ and $y_1$ respectively, and so we can properly extend $h_s$ to $h_{s+1}$ by defining $h_{s+1}(x_1)=y_1$. Furthermore, any isomorphism $H$ from $Tree_\mathcal{A} (a_0)$ to $Tree_\mathcal{B} (b_0)$ that extends $h_s$ must map $x_1$ to $y_1$. Thus, there certainly exists an isomorphism $H$ from $Tree_\mathcal{A} (a_0)$ to $Tree_\mathcal{B} (b_0)$ that extends $h_s$ and agrees with $h_{s+1}$ on $x_1$.

If $\beta_{\mathcal{A}}(x)=2$, then again, it must be the case that $\beta_{\mathcal{B}}(y)=2$. So let $x_1$ and $x_2$ be the distinct pre-images of $x$ under $f$, and let $y_1$ and $y_2$ be the distinct pre-images of $y$ under $g$. If $iso_\mathcal{A}(x) = 1$, then it must be the case that $iso_\mathcal{B}(y) = 1$. Moreover, the \emph{Trees} of both $x_1$ and $x_2$ are isomorphic to the \emph{Trees} of both $y_1$ and $y_2$. So regardless of where $h_{s+1}$ maps $x_1$ and $x_2$, there will exist an isomorphism $H$ from $Tree_\mathcal{A} (a_0)$ to $Tree_\mathcal{B} (b_0)$ that extends $h_s$ and agrees with $h_{s+1}$ on both $x_1$ and $x_2$. If $iso_\mathcal{A}(x) = 0$, then it must be the case that $iso_\mathcal{B}(y) = 0$ as well. It follows from the construction that if $h_{s+1}$ maps $x_1$ to $y_1$, then $Tree_\mathcal{A} (x_1) \cong Tree_\mathcal{B} (y_1)$, and thus it must be that $Tree_\mathcal{A} (x_2) \cong Tree_\mathcal{B} (y_2)$. The reverse statement also holds if $h_{s+1}$ maps $x_1$ to $y_2$. Since the branches of $x$ (and the branches of $y$) are not isomorphic to each other, any isomorphism $H$ from $Tree_\mathcal{A} (a_0)$ to $Tree_\mathcal{B} (b_0)$ extending $h_s$ must agree with $h_{s+1}$ on $x_1$ and $x_2$. 

Thus, it is apparent that $h_{s+1}$, once defined on all elements in $tree_\mathcal{A} (a_0|s+1)$, is an isomorphism from $Tree_\mathcal{A} (a_0|s+1)$ to $Tree_\mathcal{B} (b_0|s+1)$. Furthermore, there must exist an isomorphism $H$ from $Tree_\mathcal{A} (a_0)$ to $Tree_\mathcal{B} (b_0)$ that extends $h_{s+1}$, as the branching functions and the branch isomorphism functions prevent us from ``making a mistake" throughout the construction. Hence, $h_{s+1}$ is a proper extension of $h_s$ for all stages $s$, and once $h_s(x)$ is defined at a stage $s$, it is never redefined again. Therefore, $h=lim_s h_s$ exists and is an isomorphism from $Tree_\mathcal{A} (a_0)$ to $Tree_\mathcal{B} (b_0)$.\\

\textbf{Claim 2}. The isomorphism $h$ is a computable function.\\

\emph{Proof of Claim 2}. Let $x \in tree_\mathcal{A} (a_0)$. To determine $h(x)$ we run through the stages of the construction until $h$ is defined on $x$. By the assumption that the branching function and branch isomorphism function for both structures are computable, we can easily see that the construction is computable at every stage. Therefore, we can effectively determine the image of $x$ under $h$.\\

The construction of a computable isomorphism $h$ from $exTree_\mathcal{A} (c_1)$ to $exTree_\mathcal{B} (d_1)$ is almost identical to the one presented above. The only difference is that at stage $1$, after mapping $c_1$ to $d_1$ at stage $0$, we must then determine via the branching functions whether $c_1$ and $d_1$ have non-cyclic pre-images. If they don't, then $exTree_\mathcal{A} (c_1)$ and $exTree_\mathcal{B} (d_1)$ are trivially computably isomorphic. Otherwise, we find the non-cyclic pre-images of both $c_1$ and $d_1$ (which, of course, can be done computably), then define $h_1$ as a map from the non-cyclic pre-image of $c_1$ to that of $d_1$.
\end{proof}

It is worth noting that the construction in Lemma 2.9 can be done without the explicit assumption that $\beta_\mathcal{B}$ and $iso_\mathcal{B}$ are computable. \\

We conclude this section with our main theorem, which gives a general sufficient condition for a (2,1):1 structure to be computably categorical.

\begin{thm}
Let $\mathcal{A}=(A,f)$ be a computable (2,1):1 structure without $\mathbb{Z}$-chains and with $\beta_{\mathcal{A}}$ and $iso_{\mathcal{A}}$ computable. If for each $k \in \omega$, $\mathcal{A}$ has only finitely many k-cycles, then $\mathcal{A}$ is computably categorical.
\end{thm}

\begin{proof}
Suppose that $\mathcal{A}$ is a (2,1):1 structure as described above, and $\mathcal{B}$ is a computable structure isomorphic to $\mathcal{A}$. For each $k \in \omega$, $\mathcal{A}$ has only finitely many k-cycles, and thus $A$ has only finitely many cyclic elements $c_1,...,c_n$ in those k-cycles, which we can computably identify. So we can non-uniformly and isomorphically map each of these cyclic elements $c_i$ in $A$ to a corresponding cyclic element $d_i$ in $B$ via a computable function. Then, by Lemma 2.9, we can construct a computable isomorphism $h_{i,k}$ from $exTree_\mathcal{A}(c_i)$ to $exTree_\mathcal{B}(d_i)$ for $1 \leq i \leq n$. Let $h_k=\bigcup_i h_{i,k}$. Then $h_k$ is a computable isomorphism from the k-cycles in $\mathcal{A}$ to those in $\mathcal{B}$.

Repeat the procedure above for each $k \in \omega$, and let $h = \bigcup_k h_k$. Since $\mathcal{A}$ has no $\mathbb{Z}$-chains, every element in $A$ is in some k-cycle of $\mathcal{A}$. So, $h: A \to B$ is a computable isomorphism from $\mathcal{A}$ to $\mathcal{B}$. Thus, $\mathcal{A}$ is computably categorical.
\end{proof}

\section{Examples}

In this section, we present some examples of (2,1):1 structures with various computability-theoretic properties. Our first example illustrates our point from Section 2 that computability of a (2,1):1 structure does not guarantee computability of its branching function (nor its branch isomorphism function). 

\begin{prop}
There exists a computable (2,1):1 structure $\mathcal{A}=(A,f)$ such that $\beta_\mathcal{A}$ is not computable.
\end{prop}

\begin{proof}
Our goal is to construct a computable (2,1):1 structure $\mathcal{A}$ such that $\Lambda_\mathcal{A}$ is not a computable set. So let $C$ be some computably enumerable set that contains $0$ and is not computable. Then $C$ has a partial computable characteristic function $\chi_C$ such that $\chi_C(x)=1$ if $x \in C$, and $\chi_C(x)=\uparrow$ if $x \not\in C$ (i.e.,$\chi_C$ computes forever, that is, never halts on input $x$). In stages, we build $\mathcal{A}=(A,f)$ to be a single 1-cycle.\\

\emph{Stage 0}: Let $A_0=\{0\}$, and let $f_0(0)=0$.

\emph{Stage 1}: Let $A_1=\{0,1\}$, and let $f_1(0)=0$ and $f_1(1)=0$.

\emph{Stage s+1}: Suppose we have $A_s$ and $f_s$ from stage $s$. Find the least $a$ such that

\begin{itemize}
\item $a \in I_{\mathcal{A}_{s-1}}$, and
\item $\chi_{C,s}(a)\downarrow=1$ (i.e., $\chi_C$ halts and equals 1 on input $a$ in at most $s$ steps of its computation).
\end{itemize}

If no such $a$ exists, simply extend $A_s$ to $A_{s+1}$ and $f_s$ to $f_{s+1}$ by attaching one new number (not already in $A_s$) to each number in $extree_{\mathcal{A}_s}(0|s)$. Then move on to the next stage.

If such an $a$ does exist, take the least number $x_0$ not already in $A_s$ and define $f_{s+1}(x_0)=a$. Denote the level of $x_0$ in $extree_{\mathcal{A}_s}(0)$ by $l$, and extend the \emph{Tree} of $x_0$ by attaching $s-l$ new numbers $(x_1, x_2,...,x_{s-l})$ to $x$ such that $f_{s+1}(x_i)=x_{i-1}$ for $1 \leq i \leq s-l$. Then extend $exTree_{\mathcal{A}_s}(0,s)$ to $exTree_{\mathcal{A}_s}(0,s+1)$ by attaching one new number to each number in $extree_{\mathcal{A}_s}(0|s)$. Define $A_{s+1}=A_s \cup extree_{\mathcal{A}_s}(0|s+1)$ and continue extending $f_s$ to $f_{s+1}$ accordingly. Then move on to the next stage.

Finally, let $A = \bigcup_s A_s$ and $f=\bigcup_s f_s$. This completes the construction of $\mathcal{A}$. We must now verify two claims.\\

\textbf{Claim 1}: $\mathcal{A}=(A,f)$ is a computable (2,1):1 structure.\\

\emph{Proof of Claim 1}: We have that $A=\omega$, and is thus clearly computable. To compute $f(x)$, we simply run through the construction until we reach the stage $s$ where $x$ appears, and then determine $f_s(x)$. Due to the construction, once $f_s$ is defined on an element, we never redefine it at a later stage. Thus, $f_s(x)=f(x)$, and $f$ is computable. Therefore, $\mathcal{A}$ is computable.

To see that $\mathcal{A}$ is a (2,1):1 structure, first observe that $0$ has exactly two pre-images, $0$ and $1$. Also note that at every stage, we extend the exclusive tree of $0$ by one level, so every element has at least one pre-image. The only instance where an element is given an additional pre-image is if it had exactly one pre-image, so no element has more than two pre-images. Thus, every element either has exactly one or exactly two pre-images, making $\mathcal{A}$ a (2,1):1 structure.\\

\textbf{Claim 2}: The branching function is not computable.\\

\emph{Proof of Claim 2}: Observe that $x \in C$ if and only if $x$ has two pre-images, which is if and only if $\beta_\mathcal{A}(x)=2$. Thus, $C$ is computable if and only if $\beta_\mathcal{A}$ is computable. However, $C$ is not a computable set by assumption. Therefore, $\beta_\mathcal{A}$ cannot be computable.
\end{proof}

The following example demonstrates that computability of the branching function does not imply computability of the branch isomorphism function.

\begin{prop}
There exists a computable (2,1):1 structure $\mathcal{A}=(A,f)$ such that $\beta_\mathcal{A}$ is computable but $iso_\mathcal{A}$ is not computable.
\end{prop}

\begin{proof}
We wish to construct a computable (2,1):1 structure $\mathcal{A}$ such that $\beta_\mathcal{A}$ is computable, but no computable function $\varphi_e$ computes the branch isomorphism function $iso_\mathcal{A}$. We will accomplish this by building $\mathcal{A}$ using a standard priority argument to ensure that for all $e \in \omega$, the following requirement $P_e$ is satisfied:

\begin{center}
$P_e: \varphi_e \not= iso_\mathcal{A}$
\end{center}

We start with an effective enumeration of all partial computable functions $\{\varphi_e\}_{e \in \omega}$. Our desired structure $\mathcal{A}=(A,f)$ will again be a single 1-cycle, which we will construct in stages as follows.\\

\emph{Stage 0}: Define $A_0=\{0,2\}$, $f_0(0)=0$ and $f_0(2)=0$.

\emph{Stage s+1}: Suppose we have $A_s$ and $f_s$ from the previous stage. Let $M_s$ denote the lowest level of the exclusive tree of 0 at the end of stage $s$, i.e., $M_s$ is the unique number such that $extree_{\mathcal{A}_s}(0|M_s) \not= \emptyset$, and for all $n>M_s$, $extree_{\mathcal{A}_s}(0|n) = \emptyset$.

First, assign $\varphi_e$ to level $M_s$, where $e$ is the least number such that $\varphi_e$ has not been assigned to a level of the exclusive tree at a previous stage. Let $L_j$ denote the level of $extree_{\mathcal{A}_s}(0)$ that $\varphi_j$ has been assigned to, so $L_e=M_s$. (We will only assign a partial computable function to a level of the exclusive tree that contains only elements with two pre-images.) Then, find the least $i \leq e$ such that:

\begin{itemize}
\item $\varphi_{i,s}(x)\downarrow=1$ for some $x \in L_i$, and
\item $P_i$ has not yet received attention.
\end{itemize}

If no such $i$ exists, extend $A_s$ to $A_{s+1}$, and $f_s$ to $f_{s+1}$, by attaching two unused even numbers as pre-images to every number in $extree_{\mathcal{A}_s}(0|M_s)$. Set $M_{s+1}=M_s +1$ and go on to the next stage.

If such an $i$ exists, we say that $P_i$ \emph{requires attention}. If $x \not \in extree_{\mathcal{A}_s}(0|M_s)$, let $x_1$ and $x_2$ be the distinct pre-images of $x$, with $x_1 < x_2$. Attach two unused odd numbers to every element that is in both $tree_{\mathcal{A}_s}(x_1)$ and $extree_{\mathcal{A}_s}(0|M_s)$, and attach two unused even numbers to every element that is in both $tree_{\mathcal{A}_s}(x_2)$ and $extree_{\mathcal{A}_s}(0|M_s)$. If $x \in extree_{\mathcal{A}_s}(0|M_s)$, then simply attach one unused odd number and one unused even number to $x$. In either case, repeat the procedure for every number in the same level of the exclusive tree of $0$ as $x$. Now, the $(M_s+1)^{th}$ level of the exclusive tree of $0$ is complete.

Then, to every odd number in $extree_{\mathcal{A}_s}(0|M_s+1)$, attach exactly one unused even number. To every even number in $extree_{\mathcal{A}_s}(0|M_s+1)$, attach exactly two unused even numbers. This completes the $(M_s+2)^{nd}$ level of the exclusive tree of $0$. Let\\
$$A_{s+1}=A_s \cup extree_{\mathcal{A}_s}(0|M_s+1) \cup extree_{\mathcal{A}_s}(0|M_s+2),$$ 
extend $f_s$ to $f_{s+1}$ as described above, and set $M_{s+1}=M_s+2$. At this point, $P_i$ has \emph{received attention} and we move on to the next stage.

This ends the construction. Let $A = \bigcup_s A_s$ and $f=\bigcup_s f_s$. We must now prove the following two claims.\\

\textbf{Claim 1}: The structure $\mathcal{A}=(A,f)$ is a computable (2,1):1 structure with $\beta_\mathcal{A}$ computable.\\

\emph{Proof of Claim 1}: By construction, $A=\omega$. To compute $f(a)$, we simply run through the construction until we reach the stage $s$ where $a$ appears, then determine $f_s(a)$. Due to the construction, once $f_s$ is defined on an element, we never redefine it at a later stage. Thus, $f_s(a)=f(a)$, and $\mathcal{A}$ is computable. 

To see that $\mathcal{A}$ is a (2,1):1 structure, observe that every even number has exactly two pre-images, and every odd number has exactly one pre-image. This also proves that $\beta_\mathcal{A}$ is a computable function.\\

\textbf{Claim 2}: The branch isomorphism function is not computable.\\

\emph{Proof of Claim 2}: We prove by induction that each requirement $P_e$ is satisfied. At stage $1$, $\varphi_0$ is assigned to level $L_0=M_0=1$ of the exclusive tree of $0$, so $\varphi_0$ is assigned to the single number $2$. If $\varphi_{0,s}(2)\downarrow=1$ for some stage $s$, then $P_0$ would require attention at stage $s$. However, due to the construction, the isomorphism on the branches of 2 would be ruined at stage $s$, and thus $iso_{\mathcal{A}_t}(2)=0$ for all stages $t>s$, and $iso_\mathcal{A}(2)=0 \not=1=\varphi_0(2)$. Hence, $P_0$ is satisfied. Otherwise, $\varphi_{0,s}(2)\not=1$ for any stage $s$ and thus $\varphi_0(2)\not=1$. But the only requirement that can ruin the isomorphism on the branches of $2$ is $P_0$. (Any requirement receiving attention only ruins the isomorphism on the branches of the elements in its assigned level, due to the symmetry of the construction.) Thus, for all stages $s$, $iso_{\mathcal{A}_s}(2)=1$, which means that $iso_{\mathcal{A}}(2)=1\not=\varphi_{0,s}(2)$. Again, $P_0$ is satisfied.

Now suppose that for all $i<e$, $P_i$ is satisfied. At stage $e+1$, $\varphi_e$ is assigned to some level $L_e$. If there do not exist a stage $s$ and a number $x$ in level $L_e$ such that $\varphi_{e,s}(x)\downarrow=1$, then $P_e$ is satisfied since $iso_\mathcal{A}(x)=1$ for all $x \in extree_{\mathcal{A}}(0|L_e)$ by the same argument as before. Otherwise, let $t$ be the first stage at which for all $i<e$, $P_i$ does not require attention and $\varphi_{e,t}(x)\downarrow=1$ for some $x$ in level $L_e$. Then at stage $t$, $P_e$ would require attention, and the construction would ensure that $iso_{\mathcal{A}_r}(x)=0$ for all stages $r>t$. So, $\varphi_e(x) \not= iso_{\mathcal{A}}(x)$ and again, $P_e$ would be satisfied. Therefore, all requirements are satisfied and $iso_\mathcal{A}$ is not computable.
\end{proof}

The next example illustrates what can go wrong if we relax one of the conditions in Theorem 2.10, and allow a (2,1):1 structure to have infinitely many K-cycles of one size.

\begin{prop}
There exists a computable (2,1):1 structure $\mathcal{A}=(A,f)$ with no $\mathbb{Z}$-chains such that $\beta_\mathcal{A}$ and $iso_\mathcal{A}$ are computable, but $\mathcal{A}$ is not computably categorical.
\end{prop}

\begin{proof}
We shall first present a computable (2,1):1 structure $\mathcal{A}=(A,f)$ with the desired properties, and then construct a computable isomorphic structure $\mathcal{B}$ that is not computably isomorphic to $\mathcal{A}$. 

Let $\mathcal{A}=(A,f)$ be the (2,1):1 structure where $A=\omega$ and $f:A \to A$ is defined as follows:

$$f(x)=
		\begin{cases}
			x &\text{if } x=0 \text{ or } x \text{ is odd},\\
			x-1 &\text{if } x \equiv 2 \text{ (mod } 4),\\
			\frac{x}{2} &\text{otherwise.}
		\end{cases}$$
\begin{center}
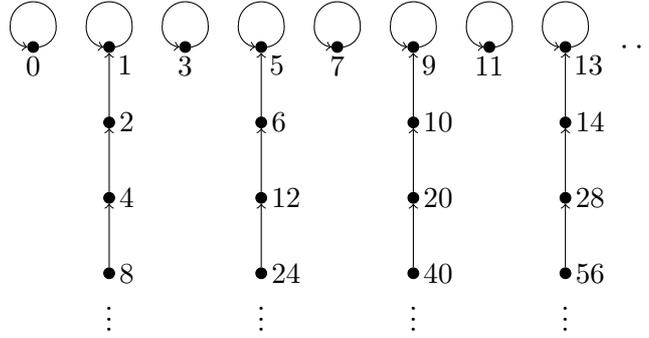

\captionof{figure}{The directed graph of $\mathcal{A}$.}
\begin{tikzpicture}[inner sep=0.5mm, vertex/.style={circle,draw,fill=black}, every label/.style={black, font=\small}]

\node[vertex](A) at (1,0)[label=below:0]{};
\node[vertex](B) at (2,0)[label=below right:1]{};
\node[vertex](C) at (3,0)[label=below:3]{};
\node[vertex](D) at (4,0)[label=below right:5]{};
\node[vertex](E) at (5,0)[label=below:7]{};
\node[vertex](F) at (6,0)[label=below right:9]{};
\node[vertex](G) at (7,0)[label=below:11]{};
\node[vertex](H) at (8,0)[label=below right:13]{};
\node at (9,0) {$\cdots$};

\node[vertex](B1) at (2,-1)[label=right:2]{};
\node[vertex](B2) at (2,-2)[label=right:4]{};
\node[vertex](B3) at (2,-3)[label=right:8]{};
\node(B4) at (2,-3.5){$\vdots$};

\node[vertex](D1) at (4,-1)[label=right:6]{};
\node[vertex](D2) at (4, -2)[label=right:12]{};
\node[vertex](D3) at (4, -3)[label=right:24]{};
\node(D4) at (4, -3.5){$\vdots$};

\node[vertex](F1) at (6,-1)[label=right:10]{};
\node[vertex](F2) at (6,-2)[label=right:20]{};
\node[vertex](F3) at (6,-3)[label=right:40]{};
\node(F4) at (6,-3.5){$\vdots$};

\node[vertex](H1) at (8,-1)[label=right:14]{};
\node[vertex](H2) at (8,-2)[label=right:28]{};
\node[vertex](H3) at (8,-3)[label=right:56]{};
\node(H4) at (8,-3.5){$\vdots$};

\draw[-] (A) to [bend right=45](1.3,0.3);
\draw[-] (1.3,0.3) to [bend right=45](1,0.6);
\draw[-] (1,0.6) to [bend right=45](0.7,0.3);
\draw[->] (0.7,0.3) to [bend right=45](A);

\draw[-] (B) to [bend right=45](2.3,0.3);
\draw[-] (2.3,0.3) to [bend right=45](2,0.6);
\draw[-] (2,0.6) to [bend right=45](1.7,0.3);
\draw[->] (1.7,0.3) to [bend right=45](B);

\draw[-] (C) to [bend right=45](3.3,0.3);
\draw[-] (3.3,0.3) to [bend right=45](3,0.6);
\draw[-] (3,0.6) to [bend right=45](2.7,0.3);
\draw[->] (2.7,0.3) to [bend right=45](C);

\draw[-] (D) to [bend right=45](4.3,0.3);
\draw[-] (4.3,0.3) to [bend right=45](4,0.6);
\draw[-] (4,0.6) to [bend right=45](3.7,0.3);
\draw[->] (3.7,0.3) to [bend right=45](D);

\draw[-] (E) to [bend right=45](5.3,0.3);
\draw[-] (5.3,0.3) to [bend right=45](5,0.6);
\draw[-] (5,0.6) to [bend right=45](4.7,0.3);
\draw[->] (4.7,0.3) to [bend right=45](E);

\draw[-] (F) to [bend right=45](6.3,0.3);
\draw[-] (6.3,0.3) to [bend right=45](6,0.6);
\draw[-] (6,0.6) to [bend right=45](5.7,0.3);
\draw[->] (5.7,0.3) to [bend right=45](F);

\draw[-] (G) to [bend right=45](7.3,0.3);
\draw[-] (7.3,0.3) to [bend right=45](7,0.6);
\draw[-] (7,0.6) to [bend right=45](6.7,0.3);
\draw[->] (6.7,0.3) to [bend right=45](G);

\draw[-] (H) to [bend right=45](8.3,0.3);
\draw[-] (8.3,0.3) to [bend right=45](8,0.6);
\draw[-] (8,0.6) to [bend right=45](7.7,0.3);
\draw[->] (7.7,0.3) to [bend right=45](H);

\draw[->] (B1) to (B);
\draw[->] (B2) to (B1);
\draw[->] (B3) to (B2);

\draw[->] (D1) to (D);
\draw[->] (D2) to (D1);
\draw[->] (D3) to (D2);

\draw[->] (F1) to (F);
\draw[->] (F2) to (F1);
\draw[->] (F3) to (F2);

\draw[->] (H1) to (H);
\draw[->] (H2) to (H1);
\draw[->] (H3) to (H2);

\end{tikzpicture}
\end{center}

This structure $\mathcal{A}$ is easily seen to be a computable (2,1):1 structure. We can also see that $\beta_\mathcal{A}(x)=2$ if $x \equiv 1$ (mod 4), and $\beta_\mathcal{A}(x)=1$ otherwise. So the function $\beta_\mathcal{A}$ is clearly computable as well. The branch isomorphism function is trivially computable, since $iso_\mathcal{A}(x)=0$ for all $x \in \Lambda_\mathcal{A}$. Finally, $\mathcal{A}$ is composed entirely of 1-cycles, and thus contains no $\mathbb{Z}$-chains. Thus, $\mathcal{A}$ has all of the desired properties.

Let $H= \{e: \varphi_e(e)\downarrow\}$ denote the halting set. We build an isomorphic copy $\mathcal{B} = (B,g)$ in stages as follows.\\

\emph{Stage 0}: Let $B_0=\{0\}$ and let $g_0(0)=0$.

\emph{Stage s+1}: Suppose we are given $B_s$ and $g_s$ from stage $s$. Find the least $e \leq s$ such that:\\

\begin{itemize}
\item $2e \in I_{\mathcal{B}_s}$
\item $g_s(2e)=2e$, and
\item $\varphi_{e,s}(e)\downarrow$
\end{itemize}

If no such $e$ exists, extend $B_s$ to $B_{s+1}$ and $g_s$ to $g_{s+1}$ by defining $g_{s+1}(2(s+1))=2(s+1)$. This adds a 1-cycle to $\mathcal{B}_s$. Also, extend any existing degenerate trees that are attached to a 1-cycle by adding an unused odd number to the end of each one. Go on to the next stage.

If such an $e$ exists, extend $B_s$ to $B_{s+1}$ and $g_s$ to $g_{s+1}$ in the following manner. Attach a degenerate tree of height $s$, composed entirely of unused odd numbers, to the 1-cycle containing $2e$. Then, extend all existing degenerate trees attached to a 1-cycle by adding an unused odd number to the end of each one. Add a new 1-cycle by defining $g_{s+1}(2(s+1))=2(s+1)$. Then go on to the next stage. 

At the end of the construction, let $B = \bigcup_s B_s$, and $g=\bigcup_s g_s$. It is easy to see that $\mathcal{B}$ is a computable (2,1):1 structure. Also, observe that $x \in \Lambda_\mathcal{B}$ if and only if $\frac{x}{2} \in H$. Since $H$ is not computable, both $H$ and $\overline H$ are infinite, which means that $\mathcal{B}$ has infinitely many 1-cycles with degenerate trees attached and infinitely many 1-cycles with empty trees attached, as does $\mathcal{A}$. Thus, $\mathcal{A} \cong \mathcal{B}$.

However, $\mathcal{B}$ cannot be computably isomorphic to $\mathcal{A}$. This is because $\Lambda_\mathcal{A}$ is computable but $\Lambda_\mathcal{B}$ is not, as the computability of $\Lambda_\mathcal{B}$ would imply the computability of $H$. A computable isomorphism from $A$ to $B$ would preserve the computability of the split hair set, so no such computable isomorphism can exist. Therefore, $\mathcal{A}$ is not computably categorical.
\end{proof}

Note that the conditions in Theorem 2.10 are sufficient for computable categoricity, but not necessary. By Theorem 1.3, it is possible for a (2,1):1 structure to have infinitely many K-cycles of the same size $K$ and still be computably categorical.

\end{document}